\documentclass[reqno]{amsart}

\usepackage[T1]{fontenc}
\usepackage[utf8]{inputenc}
\usepackage{amssymb,wasysym}
\usepackage{amsfonts}
\usepackage{amsmath,mathtools}
\usepackage{mathrsfs}
\usepackage{graphicx}
\usepackage{color}
\usepackage{hyperref}
\usepackage{esint}
\usepackage{verbatim}
\usepackage{bbm}
\usepackage{tikz}
\usepackage{enumitem}
\newcommand{\vecA}{\mathbf{A}}

\newcommand{\vecC}{\mathbf{C}}
\newcommand{\vecD}{\mathbf{D}}

\newcommand{\vecF}{\mathbf{F}}
\newcommand{\vecG}{\mathbf{G}}
\newcommand{\vecH}{\mathbf{H}}

\newcommand{\vecM}{\mathbf{M}}
\newcommand{\vecN}{\mathbf{N}}

\newcommand{\vecU}{\mathbf{U}}

\newcommand{\vecX}{\mathbf{X}}
\newcommand{\vecY}{\mathbf{Y}}

\newcommand{\vecd}{\mathbf{d}}
\newcommand{\vece}{\mathbf{e}}

\newcommand{\vecg}{\mathbf{g}}

\newcommand{\vecu}{\mathbf{u}}
\newcommand{\vecv}{\mathbf{v}}
\newcommand{\vecw}{\mathbf{w}}
\newcommand{\vecx}{\mathbf{x}}
\newcommand{\vecy}{\mathbf{y}}
\newcommand{\vecz}{\mathbf{z}}

\newcommand{\vecOmega}{\boldsymbol{\Omega}}

\newcommand {\Div}  {\mbox{\rm div\,}}

\newtheorem{theorem}{Theorem}[section]
\newtheorem{definition}[theorem]{Definition}
\newtheorem{remark}{Remark}

\begin{document}

\title[Controllability and kinetic limit of spherical particles]{Controllability and kinetic limit of spherical particles immersed in a viscous fluid}
\maketitle

\author{M.~Zoppello\footnote{Politecnico di Torino, Corso Duca degli Abruzzi, 24, 10129 Torino \email{marta.zoppello@polito.it}} \quad H.~Shum\footnote{{Department of Applied Mathematics, University of Waterloo, Waterloo, Ontario,  Canada N2L 3G1 \email{henry.shum@uwaterloo.ca}}} and M.~Morandotti\footnote{{Politecnico di Torino, Corso Duca degli Abruzzi, 24, 10129 Torino \email{marco.morandotti@polito.it}}
}}




\begin{abstract}
This paper deals with systems of spherical particles immersed in a viscous fluid. Two aspects are studied, namely the controllability of such systems, with particular attention to the case of one active particle and either one or two passive ones, and the kinetic limit of such systems as the number of particles diverges. The former issue is tackled in the framework of geometric control theory, whereas the latter resorts to Boltzmann-type formulations of the system of interacting particles.
\end{abstract}


\allowdisplaybreaks

\section{Introduction}
In recent years, there has been a growing interest in 
manipulating and controlling the motion of particles in
the low Reynolds number regime. Indeed, understanding
what is physically possible and how to effectively achieve the desired results is important both for overcoming chaotic effects (see \cite{Hocking,JTWG}) and in the
development of micromachines for specific tasks \cite{BFT,BT}. 
Some mechanisms for transporting large collections of particles in microfluidic devices include using pressure-driven fluid flow through channels, harnessing electrokinetic effects, acoustic streaming~\cite{chakraborty_microfluidic_2010, wu_acoustofluidic_2019}, applying magnetic fields to magnetic particles, and operating optical tweezers~\cite{ghadiri_microassembly_2012, bradac_nanoscale_2018}. Of these, magnetic fields and optical tweezers are particularly useful for precise manipulation of individual particles.

In this paper, we study the problem of controlling the motion of passive microparticles that are advected with the fluid flow generated by actively controlled ones. Inertialess hydrodynamics is notorious for its time-reversibility constraint, which leads to the well-known ``Scallop Theorem'' \cite{Purcell}.
As proved in \cite{ZMG,AZN}, one way to overcome the constraint and achieve net displacements of active particles undergoing periodic, reciprocal motion is to couple two or more such particles and exploit their hydrodynamic interaction. Inspired by the ideas in \cite{Walker}, and using the geometric control tools already exploited in \cite{MMSZ}, in the first part of the paper we prove total controllability (in the framework of geometric control theory \cite{Agrachev,Coron}) of systems consisting of one active and either one or two passive particles (see also \cite{SZAM2024}). That is, such particles can be moved from arbitrary initial positions to arbitrary final positions in unbounded three-dimensional space, provided that the particles are sufficiently far apart from each other so that the far-field approximation is valid.

The problem of controlling a large number of particles immersed in a viscous fluid has a strong connection with the studies on collective motion of particles or agents. Here, the main goal is to explore how simple binary interaction forces can potentially lead to the emergence of a global behavior in the form of specific patterns \cite{CPT}. In many non-controlled systems, the formation of a specific pattern is conditional on the initial data, thus it is interesting to investigate whether an external controller can intervene on the system towards pattern formation, also in those situations where this phenomenon is not the result of autonomous self-organization.

In this paper, we suppose that the controller is able to directly assign the values of the coordinates of some of the particles, regarded as active particles. For example, the controlled ones could be moved by optical tweezers or externally imposed magnetic forces~\cite{khalil_wireless_2012} in such a way that their velocities can be considered as control functions.
This would allow us to use a minimal controlling apparatus to manipulate passive particles. In the second part of the paper we study the situation when the overall number of particles goes to infinity, preserving the ratio of controlled to passive ones. This is the case when the number of controls should be a fraction of the total number of particles in order to be effective. We develop a Boltzmann-type control approach following the ideas recently presented in \cite{AlbiFornasier16,AFK}. This yields an approximation of the mean-field dynamics by means of an iterative sampling of two-agent microscopic dynamics (binary dynamics). This same principle allows us to generate control signals for the mean-field model by means of solving optimal control problems associated to the binary dynamics.

\section{Preliminaries}
Let us consider a driftless affine control system, for $\vecz\in\mathbb{R}^n$, of the form
\begin{equation}\label{001}
\dot \vecz=\sum_{\ell=1}^m\vecg^\ell(\vecz)u_\ell
\end{equation}
where $m<n$ and, for $\ell=1,\ldots,m$, the vector fields $\vecg^\ell$ belong to $C^\infty(\mathbb{R}^n,\mathbb{R}^n)$ and $t\mapsto u_\ell(t)$ are real-valued measurable functions. We will denote by $\vecu=(u_1,\ldots,u_m)$ the collection of the controls, simply referred to as \emph{the control}.
\begin{definition}[{\cite{Coron}}]\label{def_full_contr}
Let $\Omega$ an open set of $\mathbb{R}^n$. The driftless affine control system \eqref{001} is said to be \textit{controllable} in $\Omega$ if for any initial $\vecz^\circ\in\Omega$ and final $\vecz^f\in\Omega$ there exist a time $T>0$ and a measurable bounded control $\vecu\colon [0,T]\to\mathbb{R}^m$ such that $\vecz(0)=\vecz^\circ$ and $\vecz(T)=\vecz^f$, where $\vecz \colon [0,T]\to\Omega$ is the unique solution to \eqref{001} with control $\vecu$.
\end{definition}
We now recall the well-known Chow-Rashevskii Theorem, which gives sufficient conditions for the controllability of system \eqref{001}.
\begin{theorem}[{Chow-Rashewskii, \cite[Theorem~5.9]{Agrachev}}]
\label{Chow}
Let $\Omega\subseteq \mathbb{R}^n$ be an open connected set.
Assume that the vector fields $\vecg^1,\ldots,\vecg^m\in C^\infty(\Omega,\mathbb{R}^n)$ generate a Lie algebra $\mathfrak{Lie}\{\vecg^1,\ldots,\vecg^m\}$ that satisfies the \emph{Lie algebra rank condition}
\begin{equation}\label{LARC}
\dim\big(\mathfrak{Lie}\{\vecg^1(\vecz),\ldots,\vecg^m(\vecz)\}\big)=n\qquad\text{for all } \vecz\in\Omega.
\end{equation}
Then system \eqref{001} is controllable in $\Omega$.
\end{theorem} 
\section{The model}
We introduce the general setup for a system of spherical particles immersed in a viscous fluid. Of these,~$N$ are \emph{active particles}, whose velocities are directly controlled, and~$M$ are \emph{passive particles}, whose motions are determined by the interaction with the active ones.

We suppose that both active and passive particles have the same radius $a_i^A=a_j^P=a>0$. 
For every $i=1,\ldots,N$, we denote by $
\vecx_i(t)$ the position in space at time $t$ of the $i\textsuperscript{th}$ active particle and by $
\dot{\vecx}_i(t)$ its velocity at time $t$.
Analogously, for every $j=1,\ldots,M$, we denote by $
\vecy_j(t)$ and $
\dot{\vecy}_j(t)$ the position and velocity, respectively, at time $t$ of the $j$\textsuperscript{th} passive particle. 

We adopt the convention that indices $i, i'$ are associated with active particles and indices $j, j'$ are associated with passive particles. Since active and passive particles are physically identical apart from 
imposing controls on the active 
ones, many expressions for individual particles and pair-wise interactions take the same form for both classes of particles.
In this work, we use superscripts $\star,\bullet\in\{A,P\}$ for quantities that apply equally to active and passive particles, and we use $\alpha$ and $\beta$ for the indices of particles from the corresponding population of either active or passive particles. It is assumed that $\alpha,\beta = 1,\ldots,N$ for indices of active particles and $\alpha,\beta = 1,\ldots,M$ for indices of passive particles.

Assuming that the Reynolds number is small enough that inertial effects may be neglected, the fluid flow $\vecv$ is governed by the equations of incompressible Stokes flow and is assumed to vanish at infinity and to satisfy no-slip boundary conditions on the surfaces of the particles, namely,
\begin{equation*}
    \vecv(\vecz) = \begin{cases}
        \vecU_i^A + \vecOmega_i^A\times (\vecz - \vecx_i), & \text{for } \lvert\vecz-\vecx_i\rvert =a,\; i=1,\ldots,N,\\
        \vecU_j^P + \vecOmega_j^P\times (\vecz - \vecy_j), & \text{for } \lvert\vecz-\vecy_j\rvert =a,\; j= 1,\ldots,M,
    \end{cases} 
\end{equation*}
where $\vecU_\alpha^\star$ and $\vecOmega_\alpha^\star$ are the translational and rotational velocities, respectively, of particle $\alpha$ from population $\star \in \{A,P\}$. 

For small, neutrally buoyant particles, their inertia may also be neglected and the total forces and torques (due to the combination of hydrodynamics and external effects) are zero for each particle.
We denote by $
\vecF_i^A(t)$ the force that the $i$\textsuperscript{th} active particle exerts, at time $t$, on the surrounding fluid due to an external effect that we impose to steer its motion, and we assume that all passive particles are force-free. All particles, whether active or passive, are assumed to be torque-free.

By the linearity of the Stokes equations, the relationship between active forces and the velocities of the particles, in the absence of background flows, is generically described by~\cite{Kim1991}
\begin{equation}\label{eq:mobilityequation}
 \begin{pmatrix}
        \vecU_\alpha^\star \\ \vecOmega_\alpha^\star
    \end{pmatrix} = \sum_{i=1}^N
    \begin{pmatrix}
        \vecM_{\alpha i}^{\star A} \\ \vecN_{\alpha i}^{\star A}
    \end{pmatrix} \vecF_i^A,
\end{equation}
for $\star\in\{A,P\}$ and  $\alpha=1,\ldots,N$ if $\star=A$ and $\alpha=1,\ldots,M$ if $\star=P$. 
The 
$3\times3$ matrix $\vecM_{\alpha i}^{\star A}$ is the mobility tensor for the translational velocity of particle $\alpha$ of type $\star$ due to the force on active particle $i$. Similarly, the $3\times3$ matrix $\vecN_{\alpha i}^{\star A}$ is the mobility tensor for the rotational velocity of particle $\alpha$ of type $\star$ due to the force on active particle $i$. 
In general, the mobility tensors depend on the relative positions of all particles and it is not possible to obtain a closed-form expression for them. 
By symmetry of the spherical particles, the mobility tensors are independent of the orientations of the active and passive particles. We focus on the problem of controlling the particle positions without regard to their orientations, hence, the rotational velocities do not need to be considered.

Let $\vecd_{ij}^{AP}\coloneqq \vecy_j-\vecx_i$ be the displacement of the $j$\textsuperscript{th} passive particle from the $i$\textsuperscript{th} active particle and let $\vecd_{i'i}^{AA}, \vecd_{ji}^{PA}$, and $\vecd_{j'j}^{PP}$ be similarly defined for the displacements between pairs of other types of particles.
Assuming that all of the pairwise distances $r_{\alpha\beta}^{\star\bullet}\coloneqq \lvert\vecd_{\alpha\beta}^{\star\bullet}\rvert$ are large compared with 
the particle radius $a$, we use the far-field approximation for the translational mobility tensors~\cite{zuk_rotneprageryamakawa_2014, Graham2018} given by
\begin{equation}\label{MM02}
\vecM_{\alpha\beta}^{\star\bullet}=\begin{cases}
    \displaystyle \frac{1}{6\pi\mu a}\mathbb{I}, & \text{for } \star = \bullet \text{ and } \alpha = \beta,\\[2ex] 
    \displaystyle \bigg(1+\frac{a^2}{3}\nabla^2\bigg)\mathcal{G}(\vecd_{\beta,\alpha}^{\bullet\star}), & \text{otherwise,}
\end{cases}
\end{equation}
where $\star,\bullet\in\{A,P\}$,  $\alpha=1,\ldots,N$ if $\star=A$ and $\alpha=1,\ldots,M$ if $\star=P$, $\beta=1,\ldots,N$ if $\bullet=A$ and $\beta=1,\ldots,M$ if $\bullet=P$, and the function $\mathcal{G}\colon\mathbb{R}^3\setminus\{\mathbf{0}\}\to\mathbb{R}^{3\times3}_{\mathrm{sym}}$ defined by
\begin{equation*}
\mathcal{G}(\vecd) \coloneqq \frac{1}{8\pi \mu}\left(\frac{\mathbb{I}}{r}+\frac{\vecd\otimes \vecd}{r^3}\right)
\end{equation*}
is the Stokeslet, i.e., the Green's function for the Stokes equation 
corresponding to a point force 
located at the origin. Equation \eqref{MM02} is accurate up to $O\big((r_{\alpha\beta}^{\star\bullet})^{-3}\big)$ and can be extended to higher orders of accuracy by the method of reflections~\cite{Kim1991}.

Note that, to this order of accuracy, the passive particles do not affect the velocities of the active particles. Identifying $\dot{\vecx}_i$ with $\vecU_i^A$ in~\eqref{eq:mobilityequation}, we rewrite the translational velocity components in the form 
$\dot{\vecX} = \vecM^{AA}\vecF$,
where $\dot{\vecX}$ and $\vecF$ are vectors of length $3N$ containing all of the translational velocities and forces, respectively, of the active particles and $\vecM^{AA}$ is a $(3N)\times(3N)$ matrix containing the blocks $\vecM_{i'i}^{AA}$. The matrix $\vecM^{AA}$ is invertible and its inverse $(\vecM^{AA})^{-1}$ is the resistance matrix describing the linear relationship between forces applied to the fluid and the velocities of the particles, $\vecF =(\vecM^{AA})^{-1}\dot{\vecX}$.

Assuming that active particles are far apart from each other, the cross-mobilities $\vecM_{i'i}^{AA}, i'\neq i$ may be neglected and the forces are related to the velocities of the active particles by a constant drag coefficient, 
\begin{equation*}
\vecF^A_i =(\vecM^{AA})_{ii}^{-1}\dot{\vecx}_i=6\pi\mu a \dot{\vecx}_i.
\end{equation*}
The equations of motion for our system of active and passive particles are then
\begin{equation}\label{MM01}
\begin{cases}
\dot{\vecx}_i(t)=\vecu_i(t), & \text{for $i=1,\ldots,N$,}\\[2mm] 
\displaystyle \dot{\vecy}_j(t)=\sum_{i=1}^N 
\vecM_{ji}^{PA}(t)\vecF_i^A(t) = \sum_{i=1}^N \vecG_{ji}(t)\vecu_{i}(t),
& \text{for $j=1,\ldots,M$,}
\end{cases}
\end{equation} 
where 
\begin{equation*}
\vecG_{ji}=\vecG(\vecd_{ij}^{AP}) =\frac{3a}{4}\left(\frac{\mathbb{I}}{|\vecd_{ij}^{AP}|}+\frac{\vecd_{ij}^{AP}\otimes\vecd_{ij}^{AP}}{|\vecd_{ij}^{AP}|^3}\right).
\end{equation*}
Notice that the flow field generated by active particle $i$ through the prescribed velocity $\vecu_i(t)$ is, to leading order, unaffected by the presence of the other particles.
Thus we can rewrite equation \eqref{MM01} as
\begin{equation}\label{MM01_2}
\begin{cases}
\dot\vecx_i=\vecu_i & \text{for $i=1,\ldots N$,}\\[2mm]
\displaystyle \dot{\vecy}_j=\displaystyle\frac{3a}{4}\sum_{i=1}^N\bigg(\frac{1}{|\vecd_{ij}^{AP}|}\mathbb{I}+\frac{1}{|\vecd_{ij}^{AP}|^3}\vecd_{ij}^{AP}\otimes\vecd_{ij}^{AP}\bigg)\vecu_i & \text{for $j=1,\ldots M$,}
\end{cases}
\end{equation}
or also in a more compact form as
\begin{equation}\label{quellaconlaH}
\begin{pmatrix}
\dot\vecx_1\\
\vdots\\
\dot\vecx_N\\
\dot\vecy_1\\
\vdots\\
\dot\vecy_M
\end{pmatrix}=\vecH\vecu\qquad\text{where}\qquad \vecH=\begin{pmatrix}
\mathbb{I}_{3N\times3N}\\
\vecH_{1}\\
\vdots\\
\vecH_{M}
\end{pmatrix}
\end{equation}
where $\vecH_{i}$, $i=1,\ldots,M$ are $3\times 3 N$ matrices, whose $3\times 3$ blocks are given by the juxtaposition of the matrices $\vecG_{ji}$ appearing in the right-hand side of the second equation in \eqref{MM01}.\\
The velocity of each passive particle is the sum of independent contributions from the active particles. It is therefore instructive to examine in detail the case of control by one active particle, which will be the focus in the first part of this work.

We conclude this section by introducing the following definition, which defines the regime in which the far-field approximation is valid. 
\begin{definition}\label{def:wellseparated}
Let $R>0$ be the minimum separation we wish to maintain between any two particles of either type. 
We say that an instantaneous configuration of active and passive particles is \emph{well separated}, and denote this by $$(\vecx_1(t)\ldots,\vecx_N(t),\vecy_1(t),\ldots,\vecy_M(t)) \in \mathcal{S}_R^{N+M},$$ 
if the minimum distance between any two particles
at time $t$ is strictly greater than $R$. We say that, for given measurable control maps $t\mapsto (\vecu_1(t),\ldots,\vecu_N(t))$, a solution, or trajectory, $(\vecx_1,\ldots,\vecx_N,\vecy_1,\ldots,\vecy_M)\in (AC([0,T];\mathbb{R}^3))^{N+M}$ of system \eqref{MM01_2} is \emph{well separated} if, for all times $t\in[0,T]$, the configuration $(\vecx_1(t),\ldots,\vecx_N(t),\vecy_1(t),\ldots,\vecy_M(t)) \in \mathcal{S}_R^{N+M}$. 
\end{definition}
We notice that the set $\mathcal{S}_{R}^{N+M}\subset \mathbb{R}^{3(N+M)}$ is open and connected.
\section{Two controllability results}
In this section, we focus on the case of one active particle ($N=1$) and prove controllability results for two different systems of point particles. The former is the case of one passive particle ($M=1$), the latter is the case of two passive ones ($M=2$); in both scenarios, we obtain full controllability according to Definition~\ref{def_full_contr}.
In the remainder of this section, we fix the well separation radius $R\gg 2a>0$.

\subsection{One active and one passive particle}\label{1on1}
Letting $N=M=1$ and dropping the indices for particles, system \eqref{MM01_2} becomes 
\begin{equation}\label{MM01_2_1}
\begin{cases}
\dot\vecx=\vecu, \\[2mm]
\displaystyle \dot{\vecy}=\displaystyle\frac{3a}{4} \bigg(\frac{1}{|\vecd|}\mathbb{I}+\frac{1}{|\vecd|^3}\vecd\otimes\vecd\bigg)\vecu,
\end{cases}
\end{equation}
where $\vecd(t)\coloneqq \vecy(t)-\vecx(t)$, and it is in the form of \eqref{001}.
\begin{theorem}\label{221}
For every $(\vecx^\circ,\vecy^\circ),(\vecx^f,\vecy^f)\in\mathcal{S}_R^{1+1}$, 
there exist $T>0$ and a control map $\vecu\in L^\infty([0,T];\mathbb{R}^3)$ such that system \eqref{MM01_2_1},
admits a unique solution $(\vecx,\vecy)\in AC\big([0,T];\mathcal{S}_{R}^{1+1}\big)$ (depending on $\vecu$) such that~$(\vecx(0),\vecy(0))=(\vecx^\circ,\vecy^\circ)$ and~$(\vecx(T),\vecy(T))=(\vecx^f,\vecy^f)$, namely the system is controllable according to Definition~\ref{def_full_contr}.
\end{theorem}
\begin{proof}
Since system \eqref{MM01_2_1} is in the form \eqref{001}, it is sufficient to prove that condition \eqref{LARC} holds (with $n=3(N+M)=6=\dim (\mathcal{S}_{R}^{1+1})$).
We start by computing the Lie brackets of the control vector fields $\vecg^\ell\coloneqq (\vece^\ell,\vecH_1^\ell)^\top$, for $\ell=1,2,3$. 
Here, $\vecH_{1}^\ell$ is the $\ell-$th column of the matrix $\vecH_1$ appearing in \eqref{quellaconlaH}.
More precisely, we have
$$\vecg^1=\left(1,0,0,\frac{3a}4 \bigg(\frac{1}{|\vecd|}+\frac{d_1^2}{|\vecd^3}\bigg),\frac{3a}4 \frac{d_1d_2}{|\vecd|^3},\frac{3a}4\frac{d_1d_3}{|\vecd|^3}\right)^\top,$$
$$\vecg^2=\left(0,1,0,\frac{3a}4 \frac{d_1d_2}{|\vecd|^3},\frac{3a}4 \bigg(\frac{1}{|\vecd|}+\frac{d_2^2}{|\vecd|^3}\bigg),\frac{3a}4\frac{d_2d_3}{|\vecd|^3}\right)^\top,$$
$$\vecg^3=\left(0,0,1,\frac{3a}4 \frac{d_1d_3}{|\vecd|^3},\frac{3a}4\frac{d_2d_3}{|\vecd|^3},\frac{3a}4 \bigg(\frac{1}{|\vecd|}+\frac{d_3^2}{|\vecd|^3}\bigg)\right)^\top,$$
where, for $k=1,2,3$, $d_k$ is the $k$th component of the vector $\vecd$, and we define the first-order Lie brackets
\[\begin{split}
\vecv^1 & \coloneqq[\vecg^2,\vecg^3]=\left(0,0,0, 0, \displaystyle \frac{3a(8|\vecd|-9a)d_3}{16|\vecd|^4}, \displaystyle -\frac{3a(8|\vecd|-9a)d_2}{16|\vecd|^4}\right)^\top, \\
\vecv^2 & \coloneqq[\vecg^1,\vecg^3]= \left(0,0,0, \displaystyle \frac{3a(8|\vecd|-9a)d_3}{16|\vecd|^4}, 0, \displaystyle -\frac{3a(8|\vecd|-9a)d_1}{16|\vecd|^4}\right)^\top, \\
\vecv^3 & \coloneqq[\vecg^1,\vecg^2]=\left(0,0,0,\displaystyle \frac{3a(8|\vecd|-9a)d_2}{16|\vecd|^4}, \displaystyle -\frac{3a(8|\vecd|-9a)d_1}{16|\vecd|^4},0\right)^\top,
\end{split}\] 
and the second-order Lie brackets
$$
\vecw^1  \coloneqq [\vecg^3,\vecg^2]= \begin{pmatrix}
0\\
0\\
0\\
 \frac{3a(32(|\vecd|^2-3d_3^2)|\vecd|^2-12a(7d_1^2+5d_2^2-19d_3^2)|\vecd|+27a^2(2d_1^2+d_2^2-7d_3^2))}{64|\vecd|^7}\\[2mm]
 -\frac{9a^2(8|\vecd|-9a)d_1d_2}{64|\vecd|^7}\\[2mm]
 \frac{9a(32|\vecd|^2-104a|\vecd|+81a^2)d_1d_3}{64|\vecd|^7}
\end{pmatrix},$$
$$\vecw^2 \coloneqq[\vecg^1,\vecv^3]= \begin{pmatrix}
0\\
0\\
0\\
 -\frac{9a(32|\vecd|^2-104a|\vecd|+81a^2)d_1d_2}{64|\vecd|^7}\\[2mm]
 \frac{3a\big(32(3d_1^2-|\vecd|^2)|\vecd|^2-12a(19d_1^2-7d_2^2-5d_3^2)|\vecd|+27a^2(7d_1^2-2d_2^2-d_3^2)\big)}{64\|\vecd|^7}\\[2mm]
 \frac{9a^2(8|\vecd|-9a)d_2d_3}{64|\vecd|^7}
\end{pmatrix},$$
$$\vecw^3 \coloneqq[\vecg^1,\vecv^2]=
\begin{pmatrix}
0\\
0\\
0\\
 -\frac{9a(32|\vecd|^2-104a|\vecd|+81a^2)d_1d_3}{64|\vecd|^7} \\[2mm]
 \frac{9a^2(8|\vecd|-9a)d_2d_3}{64|\vecd|^7} \\[2mm]
 \frac{3a\big(32(3d_1^2-|\vecd|^2)|\vecd|^2-12a(19d_1^2-5d_2^2-7d_3^2)|\vecd|+27a^2(7d_1^2-d_2^2-2d_3^2)\big)}{64|\vecd|^7}
\end{pmatrix}.$$

To check linear independence of the vectors $\{\vecg_1,\vecg_2,\vecg_3\}\subset\mathbb{R}^6$ with some of the vectors $\vecv_1,\vecv_2,\vecv_3,\vecw_1,\vecw_2,\vecw_3$, we compute the following determinants: 
\[\begin{split}
\delta_1(\vecd)& \coloneqq\det\big((\vecg^1|\vecg^2|\vecg^3|\vecw^1|\vecv^2|\vecv^3)\big)=\delta(\vecd)d_1^2\,,\\
\delta_2(\vecd)& \coloneqq\det\big((\vecg^1|\vecg^2|\vecg^3|\vecv^1|\vecw^2|\vecv^3)\big)=-\delta(\vecd)d_2^2\,,\\
\delta_3(\vecd)& \coloneqq\det\big((\vecg^1|\vecg^2|\vecg^3|\vecv^1|\vecv^2|\vecw^3)\big)=-\delta(\vecd)d_3^2\,,\\
\end{split}\]
where 
$$\delta(\vecd)\coloneqq \frac{27a^3(8|\vecd|-9a)^2(16|\vecd|^2-42a|\vecd|+27a^2)}{8192|\vecd|^{13}}.$$
Notice that the well-separatedness condition $(\vecx,\vecy)\in\mathcal{S}_{R}^{1+1}$ is equivalent to asking that $|\vecd|>R$. 
Under this condition, 
the function $\delta(\vecd)$ is always positive, so that at least one of the three functions $\delta_k$ ($k=1,2,3$) above is non-zero provided that $\vecd\neq 0$: this is condition \eqref{LARC}.
Owing to the Chow--Rashewskii Theorem~\ref{Chow} (applied with $\Omega=\mathcal{S}_{R}^{1+1}$), the system is controllable according to Definition~\ref{def_full_contr} and the theorem is proved.
\end{proof}
\subsection{One active and two passive particles}
Letting $N=1$ and $M=2$, system \eqref{MM01_2} reads
\begin{equation}\label{MMH12}
\begin{cases} 
\dot{\vecx}(t)=\vecu(t),\\[2mm]
\displaystyle \dot{\vecy}_1(t)=  \frac{3a}{4}\left(\frac{1}{|\vecd_1(t)|}\mathbb{I}+\frac{1}{|\vecd_1(t)|^3}\vecd_1(t)\otimes \vecd_1(t)\right){\vecu}(t),\\[2mm]
\displaystyle \dot{\vecy}_2(t)=  \frac{3a}{4}\left(\frac{1}{|\vecd_2(t)|}\mathbb{I}+\frac{1}{|\vecd_2(t)|^3}\vecd_2(t)\otimes \vecd_2(t)\right){\vecu}(t),
\end{cases}
\end{equation}
where 
\begin{equation}\label{linear}
\vecd_j\coloneqq \vecy_j-\vecx, \qquad \text{for $j=1,2$;}
\end{equation}
subtracting the first equation from the last two, system \eqref{MMH12} can be re-written as
\begin{equation}\label{MMH12_2}
\begin{cases} 
\dot{\vecx}(t)=\vecu(t),\\[2mm]
\displaystyle \dot{\vecd}_1(t)=  \left(\left(\frac{3a}{4}\frac{1}{|\vecd_1(t)|}-1\right)\mathbb{I}+\frac{3a}{4}\frac{1}{|\vecd_1(t)|^3}\vecd_1(t)\otimes \vecd_1(t)\right){\vecu}(t),\\[2mm]
\displaystyle \dot{\vecd}_2(t)=  \left(\left(\frac{3a}{4}\frac{1}{|\vecd_2(t)|}-1\right)\mathbb{I}+\frac{3a}{4}\frac{1}{|\vecd_2(t)|^3}\vecd_2(t)\otimes \vecd_2(t)\right){\vecu}(t).
\end{cases}
\end{equation}
Notice that the last two equations in \eqref{MMH12_2} are decoupled, and have the same structure, therefore we expect that the action of the control produces similar behavior on the two passive particles.
\begin{theorem}\label{344}
For every $(\vecx^\circ,\vecy_1^\circ,\vecy_2^\circ),(\vecx^f,\vecy_1^f,\vecy_2^f)\in\mathcal{S}_{R}^{1+2}$, there exist $T>0$ and a control map $\vecu\in L^\infty([0,T];\mathbb{R}^3)$ such that system \eqref{MMH12} admits a unique solution $(\vecx,\vecy_1,\vecy_2)\in AC\big([0,T];\mathcal{S}_{R}^{1+2}\big)$ (depending on $\vecu$), such that $(\vecx(0),\vecy_1(0),\vecy_2(0))=(\vecx^\circ,\vecy_1^\circ,\vecy_2^\circ)$ and $(\vecx(T),\vecy_1(T),\vecy_2(T))=(\vecx^f,\vecy_1^f,\vecy_2^f)$, namely the system is controllable according to Definition~\ref{def_full_contr}.
\end{theorem}
\begin{proof}
Since systems \eqref{MMH12} and \eqref{MMH12_2} can be transformed into one another through a linear transformation, proving the controllability of the former is equivalent to proving it for the latter.
Therefore, we compute the Lie brackets of the vector fields 
$\vecg^\ell \coloneqq (\vece^\ell,\vecH_1^\ell-\mathbb{I},\vecH_2^\ell-\mathbb{I})^\top$, for $\ell=1,2,3$. 

Given an arbitrary well-separated configuration of the active and two passive particles, consider a plane $\Pi$ through all three particles. Without loss of generality, we may choose a reference frame in which 
the third axis is perpendicular to the plane $\Pi$, the active particle is at the origin, and the first axis is chosen in the direction connecting the active particle with the first passive particle, that is, $\vecd_1$ points towards the positive $x$-axis. 
Using these coordinates we have
$$\vecx=(0,0,0)^\top,\qquad
\vecd_1=(d_{1,1},0,0)^\top,\qquad
\vecd_2=(d_{2,1},d_{2,2},0)^\top,$$
with $d_{1,1}>R$,  
$\lvert\vecd_2\rvert>R$, and $\lvert\vecd_1-\vecd_2\rvert>R$.



The control vector fields of the six-dimensional reduced system consisting of the last two equations in \eqref{MMH12_2} are
\[\begin{split}
\hat\vecg^1=&\, \left(
\frac{3a}4 \bigg(\frac{1}{|\vecd_1|}+\frac{d_{1,1}^2}{|\vecd_1|^3}-\frac{4}{3a}\bigg),
0,
0,
\frac{3a}4 \bigg(\frac{1}{|\vecd_2|}+\frac{d_{2,1}^2}{|\vecd_2|^3}-\frac{4}{3a}\bigg),
\frac{3a}4 \frac{d_{2,1}d_{2,2}}{|\vecd_2|^3},
0
\right)^\top, \\
\hat\vecg^2=&\, \left(
0,
\frac{3a}4 \bigg(\frac{1}{|\vecd_1|}-\frac{4}{3a}\bigg),
0,
\frac{3a}4 \frac{d_{2,1}d_{2,2}}{|\vecd_2|^3},
\frac{3a}4 \bigg(\frac{1}{|\vecd_2|}+\frac{d_{2,2}^2}{|\vecd_2|^3}-\frac4{3a}\bigg),
0
\right)^\top, \\
\hat\vecg^3=&\, \left(
0,
0,
\frac{3a}4 \bigg(\frac{1}{|\vecd_1|}-\frac{4}{3a}\bigg),
0,
0,
\frac{3a}4 \bigg(\frac{1}{|\vecd_2|}-\frac4{3a}\bigg)
\right)^\top,
\end{split}\]
We now compute the following Lie bracket of the vector fields $\hat\vecg^\ell$, $\ell=1,2,3$
\begin{equation*}
\hat\vecv^1\coloneqq[\hat\vecg^1,\hat\vecg^2]=\begin{pmatrix}
0\\
\frac{3a(8|\vecd_1|-9a)}{16|\vecd_1|^3}\\
0\\
\frac{3a d_{2,2}(9a-8|\vecd_2|)}{16|\vecd_2|^4}\\[4mm]
\frac{3a d_{2,1}((8|\vecd_1|-9a)}{16|\vecd_1|^4}\\
0
\end{pmatrix},\qquad \hat\vecv^2\coloneqq[\hat\vecg^1,\hat\vecg^3]=\begin{pmatrix}
0\\
\frac{3a(8|\vecd_1|-9a)}{16|\vecd_1|^3}\\0\\0\\
\frac{3a d_{2,1}(9a-8|\vecd_2|)}{16|\vecd_2|^4}\\0
\end{pmatrix},
\end{equation*}
\begin{equation*}
\small
\hat\vecw^1\coloneqq[\hat\vecg^1,\hat\vecv^1]=\begin{pmatrix}
0\\
\frac{3a(189a^2-228a|\vecd_1|+64|\vecd_1|^2)}{64|\vecd_1|^5}\\
0\\
-\frac{9a d_{2,1}d_{2,2}(81a^2-104a|\vecd_2|+32|\vecd_2|^2)}{64|\vecd_2|^7}\\[4mm]
\frac{3a(27a^2(7d_{2,1}^2-2d_{2,2}^2)+12a|\vecd_2|(7d_{2,2}^2-19d_{2,1}^2)+32(2d_{2,1}^4+d_{2,1}^2d_{2,2}^2-d_{2,2}^4))}{64|\vecd_2|^7}\\
0
\end{pmatrix},
\end{equation*}
\begin{equation*}
\small
\hat\vecw^2\coloneqq[\hat\vecg^2,\hat\vecv^1]=\begin{pmatrix}
\frac{3a(54 a^2|\vecd_1|^2-84 a |\vecd_1|^3+32|\vecd_1|^4)}{64|\vecd_1|^7}\\
0\\
0\\
\frac{3a(27a^2(2d_{2,1}^2-7d_{2,2}^2)-12a|\vecd_2|(7d_{2,1}^2-19d_{2,2}^2)+32|\vecd_2|^2(d_{2,1}^2-2d_{2,2}^2)}{64|\vecd_2|^7}\\[4mm]
\frac{9a d_{2,1}d_{2,2}(81a^2-104a|\vecd_2|+32|\vecd_2|^2)}{64|\vecd_2|^7}
\end{pmatrix},
\end{equation*}
\begin{equation*}
\hat\vecw^3\coloneqq[\hat\vecg^3,\hat\vecv^2]=\begin{pmatrix}
\frac{3a(54 a^2|\vecd_1|^2-84 a |\vecd_1|^3+32|\vecd_1|^4)}{64|\vecd_1|^7}\\
0\\
0\\
\frac{3a(27a^2(2d_{2,1}^2+d_{2,2}^2)-12a|\vecd_2|(7d_{2,1}^2+5d_{2,2}^2)+32|\vecd_2|^4}{64|\vecd_2|^7}\\[4mm]
\frac{9a^2d_{2,1}d_{2,2}(9a-8|\vecd_2|)}{64|\vecd_2|^7}\\
0
\end{pmatrix},
\end{equation*}
\begin{equation*}
\small
\hat\vecw^4\coloneqq[\hat\vecg^1,\hat\vecv^2]=\begin{pmatrix}
0\\
0\\
\frac{3a(189a^2|\vecd_1|^2-228a|\vecd_1|^3+64|\vecd_1|^4)}{64|\vecd_1|^7}\\
0\\
0\\
\frac{3a(27a^2(7d_{2,1}^2-d_{2,2}^2)+12a|\vecd_2|(5d_{2,2}-19d_{2,1})+32(2d_{2,1}^4+d_{2,1}^2d_{2,2,}^2-d_{2,2}^4)}{64|\vecd_2|^7}
\end{pmatrix},
\end{equation*}
\begin{equation*}
\hat\vecw^5\coloneqq[\hat\vecg^3,[\hat\vecg^2,\hat\vecg^3]]=\begin{pmatrix}
0\\
\frac{3a(27a^2|\vecd_1|^2-60a|\vecd_1|^3+32|\vecd_1|^4)}{64|\vecd_1|^7}\\
0\\
\frac{9a^2d_{2,1}d_{2,2}(9a-8|\vecd_2|)}{64|\vecd_2|^7}\\
\frac{3a(32|\vecd_2|^4-12a|\vecd_2|(5d_{2,1}^2+7d_{2,2}^2)+27a^2(d_{2,1}^2+2d_{2,2}^2)}{64|\vecd_2|^7}\\
0
\end{pmatrix},
\end{equation*}
\begin{equation*}
\small
\hat\vecw^6\coloneqq[\hat\vecg^2,[\hat\vecg^2,\hat\vecg^3]]=\begin{pmatrix}
0\\0\\
\frac{-3a(27a^2|\vecd_1|^2-60a|\vecd_1|^3+32|\vecd_1|^4)}{64|\vecd_1|^7}\\
0\\
0\\
\frac{-3a(32(d_{2,1}^2-2d_{2,2}^2)-12a|\vecd_2|(5d_{2,1}^2-19d_{2,2}^2)+27a^2(d_{2,1}^2-7d_{2,2}^2))}{64|\vecd_2|^7}
\end{pmatrix},
\end{equation*}
\begin{equation*}
\small
\begin{split}
&\hat\vecw^7\coloneqq[\hat\vecg^1,\hat\vecw^2]\\
=&\begin{pmatrix}
-\frac{9a(27a^3-54a^2|\vecd_1|+36a|\vecd_1|^2-8|\vecd_1|^3)}{16|\vecd_1|^7}\\
0\\
0\\
\frac{9ad_{2,1}(16|\vecd_2|^3(d_{2,1}^2-4d_{2,2}^2)-12a|\vecd_2|^2(6d_{2,1}^2-29d_{2,2}^2)+36a^2|\vecd_2|(3d_{2,1}^2-17d_{2,2}^2)-27a^3(2d_{2,1}^2-13d_{2,2}^2))}{32|\vecd_2|^{10}}\\[4mm]
\frac{9ad_{2,2}(16|\vecd_2|^3(4d_{2,1}^2-d_{2,2}^2)-12a|\vecd_2|^2(29d_{2,1}^2-6d_{2,2}^2)+36a^2|\vecd_2|(17d_{2,1}^2-3d_{2,2}^2)-27a^3(13d_{2,1}^2-2d_{2,2}^2))}{32|\vecd_2|^{10}}\\
0\\
\end{pmatrix}.
\end{split}
\end{equation*}
To prove that all these vector fields are linearly independent, we compute the determinant of the matrix obtained by juxtaposing them:
\[\begin{split}
\delta^{(1)}(a)\coloneqq &\, \det\big(
\hat\vecv^1|\hat\vecw^2|\hat\vecw^3|\hat\vecw^4|\hat\vecw^5|\hat\vecw^6\big) \\
=&\, \frac{6561a^6d_{2,2}p_1^{(1)}(\vecd_1)p_2^{(1)}(\vecd_1)p_3^{(1)}(\vecd_2)p_4^{(1)}(\vecd_1,\vecd_2)}{2^{25}|\vecd_1|^{13}|\vecd_2|^{19}}
\end{split}\]
where
\begin{equation*}
\small
\begin{aligned}
p_1^{(1)}(\vecd_1)\coloneqq&\,(9a-8|\vecd_1|)\\
p_2^{(1)}(\vecd_1)\coloneqq&\,(27a^2-42a|\vecd_1|+16|\vecd_1|^2)\\
p_3^{(1)}(\vecd_2)\coloneqq&\,(243a^4-702a^3|\vecd_2|+756a^2|\vecd_2|^2-360a|\vecd_2|^3+64|\vecd_2|^4)\\
p_4^{(1)}(\vecd_1,\vecd_2)\coloneqq&\,-729a^4d_{2,2}^2-64|\vecd_1|^2|\vecd_2|^2d_{2,2}^2 -18a^2\big(13d_{2,2}^2(|\vecd_1|^2+d_{2,2}^2)-5d_{2,1}^4\\
&+56|\vecd_1||\vecd_2|d_{2,2}^2+d_{2,1}^2(5|\vecd_1|^2+8d_{2,2}^2)\big)\\
&+108a^3\big(d_{2,1}^2(|\vecd_1|-|\vecd_2|)+8d_{2,2}^2(|\vecd_1|+|\vecd_2|)\big)\\
&+24a|\vecd_1| \big(d_{2,1}^2(8d_{2,2}^2+3|\vecd_1||\vecd_2|)+11(d_{2,2}^4+|\vecd_1|d_{2,2}^2|\vecd_2|)-3d_{2,1}^4\big).
\end{aligned}
\end{equation*}
Since $\delta^{(1)}(a)/a^6\to 6561d_{2,2}^4/(64|\vecd_1|^8|\vecd_2|^{13})\neq0$ as $a\to0$, by continuity, $\delta^{(1)}(a)\neq0$ for $a\ll1$, provided that $d_{2,2}\neq0$, that is, this determinant does not vanish unless the three particles are collinear. To handle the case of collinear configurations, we also compute 
$$\delta^{(2)}(a)=\det\big(\hat\vecv^1|\hat\vecv^2|\hat\vecw^1|\hat\vecw^4|\hat\vecw^2|\hat\vecw^7\big)=\frac{2187 a^6 p_1^{(2)}(\vecd_1,\vecd_2)p_2^{(2)}(\vecd_1,\vecd_2)}{2^{31}|\vecd_1|^{19}|\vecd_2|^{24}},$$
where
\begin{equation*}
\small
\begin{aligned}
p_1^{(2)}(\vecd_1,\vecd_2)\coloneqq &\, d_{2,1}|\vecd_2|^3p_1^{(1)}(\vecd_2)r_1(\vecd_1) -|\vecd_1|^2p_1^{(1)}(\vecd_1)r_2(\vecd_2),\\
p_2^{(2)}(\vecd_1,\vecd_2)\coloneqq&\, p_1^{(1)}(\vecd_2)r_1(\vecd_1)\big(d_{2,2}q_1(\vecd_1,\vecd_2)+d_{2,1}q_2(\vecd_1,\vecd_2)\big)\\
&\,+|\vecd_1|^2 p_1^{(1)}(\vecd_1)\big(3d_{2,1}d_{2,2}r_3(\vecd_2)q_1(\vecd_1,\vecd_2)+r_4(d_{2,1},d_{2,2})q_2(\vecd_1,\vecd_2)\big),
\end{aligned}
\end{equation*}
with
\begin{equation*}
\begin{aligned}
q_1(\vecd_1,\vecd_2)\coloneqq&\,6d_{2,1}d_{2,2}s_1(\vecd_1)r_3(\vecd_2)+2|\vecd_1|^2d_{2,2}s_2(\vecd_1)s_3(d_{2,1},d_{2,2}),
\\
q_2(\vecd_1,\vecd_2)\coloneqq&\, 2|\vecd_1|^2\big(s_1(\vecd_1)r_4(d_{2,2},d_{2,1})+|\vecd_1|^2d_{2,1}s_2(\vecd_1)s_3(d_{2,2},d_{2,1})\big),
\end{aligned}
\end{equation*}
and 
\begin{equation*}
\begin{aligned}
r_1(\vecd_1)\coloneqq&\,(189a^2-228a|\vecd_1|+64|\vecd_1|^2),
\\
r_2(\vecd_2)\coloneqq&\,27a^2(7d_{2,1}^2-d_{2,2}^2)+12a|\vecd_2|(5d_{2,2}^2-19d_{2,1}^2)\\
&\,+32(2d_{2,1}^4+d_{2,1}^2d_{2,2}^2-d_{2,2}^4),\\
r_3(\vecd_2)\coloneqq&\,81a^2-104a|\vecd_2|+32|\vecd_2|^2,\\
r_4(\alpha,\beta)\coloneqq&\,27a^2(7\alpha^2-2\beta^2)+12a\sqrt{\alpha^2+\beta^2}(7\beta^2-19\alpha^2)+32(2\alpha^4+\alpha^2\beta^2-\beta^4),
\end{aligned}
\end{equation*}
\begin{equation*}
\begin{aligned}
s_1(\vecd_1)\coloneqq &\,27 a^3-54a^2|\vecd_1|+36a|\vecd_1|^2-8|\vecd_1|^3, \\
s_2(\vecd_1)\coloneqq& \, 27 a^2-42a|\vecd_1|+16|\vecd_1|^2, \\
s_3(\alpha,\beta)\coloneqq&\, 27a^3(13\alpha^2-2\beta^2)-36a^2\sqrt{\alpha^2+\beta^2}(17\alpha^2-3\beta^2)\\
&\,+12a(29\alpha^4+23\alpha^2\beta^2-6\beta^4)+16\sqrt{\alpha^2+\beta^2}(\beta^4-3\alpha^2\beta^2-4\alpha^4).
\end{aligned}
\end{equation*}
Observe now that 
\begin{equation*}
\frac{\delta^{(2)}(a)}{a^6}\to\frac{2187 d_{2,1}\big(A(\vecd_1,\vecd_2)-B(\vecd_2)\big)\big(|\vecd_1|A(\vecd_1,\vecd_2)-2|\vecd_1|B(\vecd_2)+2|\vecd_2|^4\big)}{64|\vecd_1|^{10}|\vecd_2|^{15}}
\end{equation*}
as $a\to0$, where $A(\vecd_1,\vecd_2)\coloneqq|\vecd_1|(2d_{2,1}^2-d_{2,2}^2)$ and $B(\vecd_2)\coloneqq 2d_{2,1}|\vecd_2|^2$. 
When $d_{2,2}=0$, we have $A(\vecd_1,\vecd_2)=2d_{2,1}^2|\vecd_1|$ and $B(\vecd_2)=2d_{2,1}^3$, which implies that (recalling that $\vecd_1=(d_{1,1},0)$)
\begin{equation*}
\frac{\delta^{(2)}(a)}{a^6}\bigg|_{d_{2,2}=0}\to\frac{2187 (d_{1,1}-d_{2,1})^3}{16d_{1,1}^{10}d_{2,1}^{10}}.
\end{equation*}
Since this last expression vanishes only when $d_{1,1}=d_{2,1}$, which violates the condition of a minimum separation between the passive particles, we obtain that, by continuity, for $a\ll 1$, $\delta^{(2)}(a)\neq 0$ when $\delta^{(1)}(a)= 0$. 
Therefore, we have that either $\delta^{(1)}(a)$ or $\delta^{(2)}(a)$ do not vanish, proving that either one of the two sets of brackets $\{\hat\vecv^1|\hat\vecw^2|\hat\vecw^3|\hat\vecw^4|\hat\vecw^5|\hat\vecw^6\}$ and $\{\hat\vecv^1|\hat\vecv^2|\hat\vecw^1|\hat\vecw^4|\hat\vecw^2|\hat\vecw^7\}$ are linearly independent. Thus condition \eqref{LARC} is always satisfied and the theorem is proved.
\end{proof}
\section{Kinetic limit}
\label{sec:kinetic}
In this section, we investigate the behavior of the dynamics in \eqref{MM01_2} in the scenario when the number of spheres tends to infinity, preserving the ratio $N/M$ of active particles over passive ones. 
To do so, we will use a Boltzmann-type approach for dynamical systems involving binary interactions; in our case, the interaction is derived from the velocity of the passive particles.
Structurally, this is the system considered in \eqref{MM01_2_1}, which we report here, for the convenience of the reader, using the notation of \eqref{MM01}
\begin{equation}\label{MM01_2_1_G}
\begin{cases}
\dot\vecx=\vecu,\\
\dot{\vecy}=\vecG(\vecy-\vecx)\vecu.
\end{cases}
\end{equation}

Let us denote by $\vecX$ and $\vecY$ the positions of two general particles of the ensemble $\mathcal{E}^{N+M}\coloneqq\{\vecx_i\}_{i=1}^N\cup\{\vecy_j\}_{j=1}^M$ and let $\Theta_{\vecX},\Theta_{\vecY}\in\{0,1\}$ be two discrete random variables describing the probability of finding an active particle at $\vecX$ or at $\vecY$.
More specifically, we assume that both $\Theta_{\vecX}$ and $\Theta_{\vecY}$ have a Bernoulli distribution of parameter $p=N/M$ and that 
$$\mathbb{P}\{\text{$\vecX$ is active}\}=\mathbb{P}\{\Theta_{\vecX}=1\}=p
\quad\text{and}\quad
\mathbb{P}\{\text{$\vecY$ is active}\}=\mathbb{P}\{\Theta_{\vecY}=1\}=p.$$
When drawing at random two particles 
from the ensemble $\mathcal{E}^{N+M}$, we can reformulate the dynamics of \eqref{MM01_2_1_G} as follows:
\begin{subequations}\label{Kinetic_inter_main}
\begin{equation}
\label{Kinetic_inter}
\vecX^*=\vecX+\tau\Theta_{\vecX}\vecu +\tau(1-\Theta_{\vecX})\Theta_{\vecY} \vecG(\vecX-\vecY)\vecu\,,
\end{equation}
where $\tau>0$ is a characteristic time of interaction and $\vecX^*$ is the position of the particle that was at $\vecX$ after the interaction with the particle at $\vecY$.

We briefly describe the dynamics in \eqref{Kinetic_inter}.
If $\Theta_{\vecX}=1$, then $\vecX$ is occupied by an active particle and \eqref{Kinetic_inter} reads $\vecX^*=\vecX+\tau\vecu$ (regardless of the nature of $\vecY$), which is the first equation in \eqref{MM01_2_1_G}. 
If $\Theta_{\vecX}=0$ and $\Theta_{\vecY}=1$, then $\vecX$ is occupied by a passive particle and $\vecY$ by an active one; \eqref{Kinetic_inter} reads $\vecX^*=\vecX+\tau\vecG(\vecX-\vecY)\vecu$, which is the second equation in \eqref{MM01_2_1_G}.\footnote{Notice that in both cases this is an explicit Euler scheme for the dynamics of \eqref{MM01_2_1_G}.}
In the remaining case of drawing two passive particles, there is no dynamics between them.
Analogously, the particle at $\vecY$ is subject to
\begin{equation}
\label{Kinetic_inter_b}
\vecY^*=\vecY+\tau\Theta_{\vecY}\vecu +\tau(1-\Theta_{\vecY})\Theta_{\vecX} \vecG(\vecY-\vecX)\vecu\,.
\end{equation}
\end{subequations}



We now define $\mu(\vecX,t)$ to be the density of particles of $\mathcal{E}^{N+M}$ at point $\vecX\in\mathbb{R}^3$ at time $t$.
As postulated at the beginning of this section, the time evolution of $\mu$ is governed by the following Boltzmann-type equation
\begin{equation}
\label{Boltzman1}
\partial_t \mu(\vecX,t) = \mathbb{Q}_\tau[\mu,\mu](\vecX,t)\,,
\end{equation}
where $\mathbb{Q}_\tau$ is an interaction operator which accounts for the gain and loss of particles at position $\vecX$ and has the expression (see, e.g., \cite[formula (4.22)]{AFK} or \cite[formula (3.1)]{Toscani})
\begin{equation}
\label{interaction_kernel}
\begin{split}
\mathbb{Q}_\tau[\mu,\mu](\vecX,t) = \sigma \bigg\langle\!\bigg\langle\int_{\mathcal{S}_R(\vecX)} \bigg(\frac{1}{\mathcal{J}_\tau}\mu(\vecX_*,t)\mu(\vecY_*,t)-\mu(\vecX,t)\mu(\vecY,t)\bigg)\mathrm{d}\vecY\bigg\rangle\!\bigg\rangle\,.
\end{split}
\end{equation}
In \eqref{interaction_kernel} above, $\sigma>0$ denotes the interaction rate, $(\vecX_*,\vecY_*)$ are the pre-interaction positions that generate $(\vecX,\vecY)$ via the interaction rule \eqref{Kinetic_inter_main}, $\mathcal{J}_\tau$ is the Jacobian determinant of the transformation $(\vecX,\vecY)\to(\vecX^*,\vecY^*)$ in \eqref{Kinetic_inter_main} (and, of course, of the transformation $(\vecX_*,\vecY_*)\to(\vecX,\vecY)$), $\mathcal{S}_R(\vecX)=\mathbb{R}^3\setminus \overline{B}_R(\vecX)$, and the average $\langle\!\langle\cdot\rangle\!\rangle$ is taken over $\Theta_{\vecX}$ and $\Theta_{\vecY}$.



We now give the notion of weak solution to \eqref{Boltzman1}, for which we need to introduce the symbol $\mathcal{P}_1(\mathbb{R}^3)$ to denote the probability measures on $\mathbb{R}^3$ with finite first moment. 
\begin{definition}\label{def_weaksolBoltz}
A time-dependent probability density $[0,+\infty)\ni t\mapsto\mu(\cdot,t)$ is called a weak solution to the Boltzmann equation \eqref{Boltzman1} with initial datum $\mu_0\in\mathcal{P}_1(\mathbb{R}^3)$ if 
$\mu\in L^2(0,+\infty;\mathcal{P}_1(\mathbb{R}^3))$ and satisfies the weak form of \eqref{Boltzman1}, 
$$
\frac{\mathrm{d}}{\mathrm{d}t} \langle\mu,\varphi\rangle=\langle\mathbb{Q}_\tau[\mu,\mu],\varphi\rangle = \sigma \bigg\langle\!\bigg\langle \int_{\mathbb{R}^3}\int_{\mathcal{S}_R(\vecX)}\!\!(\varphi(\vecX^*)-\varphi(\vecX))\mu(\vecX,t)\mu(\vecY,t)\,\mathrm{d}\vecY\mathrm{d}\vecX\bigg\rangle\!\bigg\rangle,
$$
for all $t\in[0,+\infty)$ and all $\varphi\in C^\infty_c(\mathbb{R}^3)$, and it recovers the initial datum in the sense
\begin{equation*}
\lim_{t\to0^+} \int_{\mathbb{R}^3} \varphi(\vecX)\mu(\vecX,t)\,\mathrm{d}\vecX= \int_{\mathbb{R}^3} \varphi(\vecX)\mu_0(\vecX)\,\mathrm{d}\vecX,
\end{equation*}
for every $\varphi\in C^\infty_c(\mathbb{R}^3)$. 
\end{definition}
We are now ready to state the main result of this section, namely to obtain the equation satisfied by the density of particles $\mu$ in the limit $\sigma\to 0$ of low interaction strength and $\tau\to+\infty$ of high frequency of interactions. We will do this in the specific regime $\sigma=1/\tau$ that balances the limits of these two quantities.\footnote{For the sake of generality, we could assume that $\sigma\tau\to c$, where $c>0$ is any constant, but we choose $c=1$ for convenience. Likewise, we could consider the limits $\sigma\tau\to0$ or $\sigma\tau\to+\infty$, but either situation means that the interaction time scale is different from the frequency scale, so that we could expect less significant behaviors in the limit.}
\begin{theorem}\label{K_limit}
Let $\vecu\in L^\infty(0,+\infty;\mathbb{R}^3)$ and let $\mu_0\in\mathcal{P}_1(\mathbb{R}^3)$; for every $\tau>0$, let $\mu^\tau$ be a weak solution to the Boltzmann equation \eqref{Boltzman1} with $\sigma=1/\tau$, according to Definition~\ref{def_weaksolBoltz}.
Then for $\tau\to0$, the densities of particles $\{\mu^\tau\}_{\tau>0}$ converge pointwise, up to a subsequence, to a measure $\mu\in L^2(0,+\infty;\mathcal{P}_1(\mathbb{R}^3))$ that satisfies
\begin{equation}\label{PDE}
\partial_t \mu+\Div\bigg[\bigg(\int_{\mathcal{S}_R(\vecX)}
\mathcal{K}(\vecX,\vecY,\vecu)
\mu(\vecY,t)\,\mathrm{d}\vecY\bigg)\mu\bigg]=0
\end{equation}
with initial datum $\mu_0$\,, where 
\begin{equation}\label{K}
\mathcal{K}(\vecX,\vecY,\vecu)
=p(1-p)\vecG(\vecY-\vecX) \vecu+p\vecu\,.
\end{equation}
\end{theorem}
\begin{proof}
The proof follows the ideas of the one in \cite[Theorem 4.4]{AlbiFornasier16}, where now the binary dynamics is an affine function of the controls.

Since $\mu^\tau=\mu^\tau(\vecX,t)$ is a weak solution of \eqref{Boltzman1} with initial datum $\mu_0$\,, for every test function $\varphi\in C^\infty_c(\mathbb{R}^3)$ we have
\begin{equation}\label{weak_eq}
\frac{\mathrm{d}}{\mathrm{d} t} \int_{\mathbb{R}^3} \varphi(\vecX)\mu^\tau(\vecX,t)\,\mathrm{d}\vecX = \frac{1}{\tau}\bigg\langle\!\bigg\langle \int_{\mathbb{R}^3}\int_{\mathcal{S}_{R}(\vecX)}  \!\!\!\! (\varphi(\vecX^*)-\varphi(\vecX)) \mu^\tau(\vecX,t)\mu^\tau(\vecY,t)\,\mathrm{d}\vecY\mathrm{d}\vecX\bigg\rangle\!\bigg\rangle\,.
\end{equation}
For $\tau$ small enough, we approximate, using Taylor expansion up to second order,
\begin{equation}\label{phi_diff}
\begin{split}
\varphi(\vecX^*)-\varphi(\vecX)= &\, \nabla\varphi(\vecX)\cdot(\vecX^*-\vecX) \\
&\, +\frac{1}{2} H\varphi(\vecX)(\vecX^*-\vecX)\cdot(\vecX^*-\vecX)+o(|\vecX^*-\vecX|^2)\,.
\end{split}
\end{equation}
The boundedness of $\vecu$ implies that of the control in the right-hand side of \eqref{Kinetic_inter}, so that the second-order term $|H\varphi(\vecX)(\vecX^*-\vecX)\cdot(\vecX^*-\vecX)|=O(\tau^2)$ becomes of order $\tau$ when plugging \eqref{phi_diff} in \eqref{weak_eq}.
Thus, we obtain
\begin{equation*}
\begin{split}
&\, \frac{\mathrm{d}}{\mathrm{d}t} \int_{\mathbb{R}^3}\varphi(\vecX)\mu^\tau(\vecX,t)\,\mathrm{d}\vecX \\
=&\, \bigg\langle\!\bigg\langle \int_{\mathbb{R}^3} \nabla\varphi(\vecX)\cdot\bigg(\int_{\mathcal{S}_{R}(\vecX)} \frac{\vecX^*-\vecX}{\tau} \mu^\tau(\vecY,t)\,\mathrm{d}\vecY\bigg)\mu^\tau(\vecX,t)\,\mathrm{d}\vecX\bigg\rangle\!\bigg\rangle +O(\tau)\\
=& \int_{\mathbb{R}^3} \nabla\varphi(\vecX)\cdot\bigg(\int_{\mathcal{S}_{R}(\vecX)} \!\! \big\langle\!\big\langle \Theta_{\vecX}\vecu+(1-\Theta_{\vecX})\Theta_{\vecY}\vecG(\vecY-\vecX)\vecu \big\rangle\!\big\rangle \mu^\tau(\vecY,t)\,\mathrm{d}\vecY\bigg)\mu^\tau(\vecX,t)\,\mathrm{d}\vecX\\
&\,+O(\tau)\\
=& \int_{\mathbb{R}^3} \nabla\varphi(\vecX)\cdot\bigg(\int_{\mathcal{S}_{R}(\vecX)} \mathcal{K}(\vecX,\vecY,\vecu) \mu^\tau(\vecY,t)\,\mathrm{d}\vecY\bigg)\mu^\tau(\vecX,t)\,\mathrm{d}\vecX +O(\tau)\\
=&\, -\int_{\mathbb{R}^3} \varphi(\vecX)\,\Div\bigg[\bigg(\int_{\mathcal{S}_{R}(\vecX)} \mathcal{K}(\vecX,\vecY,\vecu) \mu^\tau(\vecY,t)\,\mathrm{d}\vecY\bigg)\mu^\tau(\vecX,t)\bigg]\,\mathrm{d}\vecX +O(\tau)\,.
\end{split}
\end{equation*}
By taking the limit as $\tau\to0$, we get
$$\frac{\mathrm{d}}{\mathrm{d}t} \int_{\mathbb{R}^3}\varphi(\vecX)\mu(\vecX,t)\,\mathrm{d}\vecX= -\int_{\mathbb{R}^3} \varphi(\vecX)\,\Div\bigg[\bigg(\int_{\mathcal{S}_{R}(\vecX)} \mathcal{K}(\vecX,\vecY,\vecu) \mu(\vecY,t)\,\mathrm{d}\vecY\bigg)\mu(\vecX,t)\bigg]\,\mathrm{d}\vecX\,,$$
which is the weak form of \eqref{PDE}.
The theorem is proved.
\end{proof}

\subsection{The optimal control problem: binary control and Boltzmann approach}
In this section, starting from a finite-dimensional optimal control problem for the case of two particles interacting according to \eqref{Kinetic_inter_main} (in particular, see the scenario studied in Section~\ref{1on1}), we derive the infinite-dimensional one, via the Boltzmann approach introduced above. 
We will give a sub-optimal solution to the infinite-dimensional optimal control problem starting from the optimal solution of the finite-dimensional one.
\subsubsection{The binary optimal control}
Let us consider two particles $\vecx,\vecy\in\mathbb{R}^3$ that interact according to the dynamics described in \eqref{Kinetic_inter_main}, where the control $\vecu$ takes values in a compact set $\mathcal{U}\subset\mathbb{R}^{3}$. 
We denote by $\vecz\coloneqq(\vecx,\vecy)\in\mathbb{R}^6$ the state of the system and we are interested in minimizing, for $\vecu\in L^\infty([0,T],\mathcal{U})$, the following finite-horizon functional
\begin{equation}\label{opt_finite_1}
\mathcal{J}(\vecu,\vecz_0)\coloneqq \int_0^T\mathscr{L}_{\bar\vecz}(\vecz(t),\vecu(t))\,\mathrm{d}t\,.
\end{equation}
In \eqref{opt_finite_1}, the cost function is
\begin{equation}\label{ellecorsivo}
\mathscr{L}_{\bar\vecz}(\vecz,\vecu)=\frac{1}{2}\lvert\bar{\vecz}-\vecz\rvert^2+\frac{\gamma}{2}\lvert\vecu\rvert^2\,,
\end{equation}
with $\bar\vecz$ a desired target state and $\gamma>0$ a weight, measuring the relative importance of the magnitude of the control. 
The point $\vecz_0=(\vecx_0,\vecy_0)\in\mathbb{R}^6$ contains the initial conditions.

For $K\in\mathbb{N}$, we discretize the time interval $[0,T]$ by dividing it into $K$ subintervals $[t_{k-1},t_{k}]$, for $k=1,\ldots,K$, of length $1/K$, with $t_0=0$ and $t_{K}=T$.
Furthermore, we assume that the control $\vecu$ can be considered constant in each subinterval, that is $\vecu(0)=\vecu_0^K$ and $\vecu(t)=\vecu_k^K$ for $t\in(t_{k-1},t_{k}]$. We can also represent this control as $\vecu(t)=\sum_{k=1}^{K} \vecu_k^K\chi_{(t_{k-1},t_{k}]}(t)$.
The discretization of the dynamics \eqref{Kinetic_inter_main} is
\begin{equation}\label{1.25}
\begin{cases}
\vecx_{k}=\vecx_{k-1}+K^{-1}[\Theta_{\vecx}\mathbb{I}+(1-\Theta_{\vecx})\Theta_{\vecy}\vecG(\vecx_{k-1}-\vecy_{k-1})]\vecu_{k-1}^K\,,
\\
\vecy_{k}=\vecy_{k-1}+K^{-1}[\Theta_{\vecy}\mathbb{I}+(1-\Theta_{\vecy})\Theta_{\vecx}\vecG(\vecy_{k-1}-\vecx_{k-1})]\vecu_{k-1}^K\,,
\end{cases}
\end{equation}
and, accordingly, we put $\vecz_k\coloneqq(\vecx_k,\vecy_k)$, for all $k=1,\ldots,K$.
We now look at the discretized functional (for $\vecu^K\coloneqq\{\vecu_0^K,\ldots,\vecu_K^K\}$ the discretization of the control $\vecu$)
\begin{equation*}
\mathcal{J}_{K}(\vecu^K,\vecz_0\big)\coloneqq 
K^{-1}\sum_{k=1}^{K}\mathscr{L}_{\bar\vecz}(\vecz_{k},\vecu_{k-1}^K)\,.
\end{equation*}
The strategy is to optimize the values $\vecu^K$ in each time interval by solving an instantaneous optimal control problem. 
For example, in the time interval $[0,t_1]$
the instantaneous control problem is further simplified to minimizing the quantity
\begin{equation}\label{J_discr}
K^{-1}\mathscr{L}_{\bar\vecz}(\vecz_1,\vecu_0^K)=\frac1{K}\bigg(\frac{1}{2}\lvert\bar{\vecz}-\vecz_{1}\rvert^2+\frac{\gamma}{2}\lvert\vecu_0^K\rvert^2\bigg)\,,
\end{equation}
with $\vecz_{1}$ given by \eqref{1.25} for $k=1$, which reads
\begin{equation*}
\begin{cases}
\vecx_1=\vecx_0+K^{-1}[\Theta_{\vecx}\mathbb{I}+(1-\Theta_{\vecx})\Theta_{\vecy}\vecG(\vecx_0-\vecy_0)]\vecu_0^K \eqqcolon \vecx_0+K^{-1}\vecA_{\vecx}\vecu_0^K\,,\\
\vecy_1=\vecy_0+K^{-1}[\Theta_{\vecy}\mathbb{I}+(1-\Theta_{\vecy})\Theta_{\vecx}\vecG(\vecy_0-\vecx_0)]\vecu_0^K \eqqcolon \vecy_0+K^{-1}\vecA_{\vecy}\vecu_0^K\,.
\end{cases}
\end{equation*}
Since $\vecz_{1}$ depends linearly on $\vecu_0^K$\,, the expression in \eqref{J_discr} is convex in $\vecu_0^K$ and the optimal control $\bar\vecu_0^K$ is obtained by stationarising it with respect to $\vecu_0^K$. 
A straightforward computation leads to
\begin{equation}\label{Dk}
\underbrace{\big( K^{-1}(\vecA_{\vecx}^2+\vecA_{\vecy}^2)+\gamma K\mathbb{I} \big)}_{\eqqcolon \vecD^K(\vecx_0,\vecy_0)}\vecu_0^K = \underbrace{\vecA_{\vecx}(\bar\vecx-\vecx_0)+\vecA_{\vecy}(\bar\vecy-\vecy_0)}_{\eqqcolon \vecC(\vecx_0,\vecy_0)}.
\end{equation}
Since $\vecA_{\vecx}$ and $\vecA_{\vecy}$ are bounded matrices, for $K$ sufficiently large, the matrix $\vecD^K(\vecx_0,\vecy_0)$ is invertible, so that we can define
\begin{equation}\label{opt_inst_control}
\bar\vecu_0^K= \bar\vecu_0^K(\vecx_0,\vecy_0)\coloneqq \Pi_{\mathcal{U}} [(\vecD^K(\vecx_0,\vecy_0))^{-1}\vecC(\vecx_0,\vecy_0)],
\end{equation}
where $\Pi_\mathcal{U}$ is the projection on the compact $\mathcal{U}$. By iterating this procedure in each time interval, we generate an optimal control in feedback form 
\begin{equation}\label{numeretto}
\bar\vecu^K\coloneqq \big\{\bar\vecu_0^K(\vecx_0,\vecy_0), \ldots,\bar\vecu_K^K(\vecx_K,\vecy_K)\big\}\,.
\end{equation}
\begin{remark}
From \eqref{Dk} and \eqref{opt_inst_control} it is emerges that the piecewise constant optimal control $\bar\vecu^K$ is, in each time interval $(t_{k-1},t_{k}]$, of order $O(1/K)$. 
When plugging it into \eqref{1.25}, we obtain that $\lvert \vecz_{k}-\vecz_{k-1}\rvert = O(1/K^2)$, which makes the first-order contribution in \eqref{phi_diff} of order $O(1/K^2)$.
If, ideally, we take $1/K=\tau$, then such a contribution would be neglected too in the proof of Theorem~\ref{K_limit}.
A way to prevent this from happening is to make the weight $\gamma$ scale with $K$ as well, for instance by requiring that $\gamma=\bar\gamma/K$.

By doing so, the matrix $\vecD^K$ in \eqref{Dk} becomes $K^{-1}(\vecA_{\vecx}^2+\vecA_{\vecy}^2)+\bar\gamma\mathbb{I}$, which, for $K$ large, is a perturbation (of order $1/K$) of $\bar\gamma\mathbb{I}$.
Thus, $(\vecD^K)^{-1}\vecC$ is of order $1$, and so are the optimal controls $\{\bar\vecu_k^K\}_{k=1}^K$ in \eqref{opt_inst_control}.
Going back to \eqref{1.25}, we now obtain that $\lvert \vecz_{k}-\vecz_{k-1}\rvert = O(1/K)$ and the proof of Theorem~\ref{K_limit} can be carried out as we presented it, provided that the controls $\bar\vecu^K$ converge to some $\bar\vecu\in L^\infty([0,T];\mathcal{U})$. 
\end{remark}
\subsubsection{The Boltzmann approach for the infinite dimensional optimal control}
For a large ensemble of particles, the microscopic optimal control problem is well approximated by the following mean-field optimal control problem
\begin{equation}\label{this_problem}
\begin{cases}
\min\big\{J(\mu,\vecu): \vecu\in L^\infty([0,T];\mathcal{U}) \big\},
\;\; \text{with $\mu$ subject to}\\[2mm]
\displaystyle \partial_t \mu+\Div\bigg[\bigg(\int_{\mathcal{S}_R(\vecX)} \mathcal{K}(\vecX,\vecY,\vecu) \mu(\vecY,t)\,\mathrm{d} \vecY \bigg)\mu\bigg]=0,
\end{cases}
\end{equation}
where the functional $J$ is defined by
\begin{equation*}
J(\mu,\vecu)\coloneqq 
\int_0^T\int_{\mathbb{R}^3}
\mathscr{L}_{\bar\vecX}(\vecX,\vecu(\vecX,t))\mu(\vecX,t)\,\mathrm{d}\vecX\mathrm{d}t\,,
\end{equation*}
with $\mathcal{K}$ given in \eqref{K} and $\mathscr{L}$ is given in \eqref{ellecorsivo}.

We propose a sub-optimal solution to the problem in \eqref{this_problem} by using a Boltzmann-type equation to model the evolution of a system of particles governed by binary interactions.
The time evolution of the density  $\mu$ is given as a balance between gains and losses of particles, due to the following constrained binary interaction
\begin{equation}
\label{opt_feedback_dyn}
\begin{cases}
\vecX^*=\vecX+K^{-1}[\Theta_{\vecX}\mathbb{I}+(1-\Theta_{\vecX})\Theta_{\vecY}\vecG(\vecX-\vecY)]\bar\vecu^K\,,\\
\vecY^*=\vecY+K^{-1}[\Theta_{\vecY}\mathbb{I}+(1-\Theta_{\vecY})\Theta_{\vecX}\vecG(\vecY-\vecX)]\bar\vecu^K\,,
\end{cases}
\end{equation}
where $(\vecX^*,\vecY^*)$ are the post-interaction states, the parameter $1/K$ measures the strength of the interaction, and the optimal control $\bar\vecu^K=\bar\vecu^K(\vecX,\vecY)$ for the discretized dynamics (see \eqref{numeretto}) is now used as the control in \eqref{opt_feedback_dyn}.

We now proceed similarly to Theorem \ref{K_limit}, making the quasi-invariant limit by supposing that the density $\mu$ satisfies a Boltzmann-type equation and the particle dynamic is the one in \eqref{opt_feedback_dyn}. 
\begin{theorem}
Assume that $\bar\vecu^K\to \bar\vecu\in L^\infty([0,T];\mathcal{U})$ pointwise as $K\to\infty$ and let $\mu_0\in\mathcal{P}_1(\mathbb{R}^3)$; let $\mu^K$ be a weak solution to the Boltzmann equation \eqref{Boltzman1} with $\sigma=K$, according to Definition~\ref{def_weaksolBoltz} (suitably adapted on $[0,T]$).
Then for $K\to\infty$, the densities of particles $\{\mu^K\}_{K}$ converge pointwise, up to a subsequence, to a measure $\mu\in L^2(0,T;\mathcal{P}_1(\mathbb{R}^3))$ that satisfies
$$
\partial_t \mu+\Div\bigg[\bigg(\int_{\mathcal{S}_R(\vecX)}
\mathcal{K}(\vecX,\vecY,\bar\vecu)
\mu(\vecY,t)\,\mathrm{d}\vecY\bigg)\mu\bigg]=0
$$
with initial datum $\mu_0$\,, where $\mathcal{K}$ is defined in \eqref{K}.
\end{theorem}
\begin{proof}
The proof is similar to that of Theorem \ref{K_limit} and is left to the reader. 
\end{proof}



\noindent\textbf{Acknowledgments}\\
MZ is a member of INdAM-GNFM; MM is a member of INdAM-GNAMPA.\\\\

\noindent\textbf{Funding}\\
HS acknowledges the support of the Natural Sciences and Engineering Research Council of Canada (NSERC), [funding reference no. RGPIN-2018-04418]. Cette recherche a \'{e}t\'{e} financ\'{e}e par le Conseil de recherches en sciences naturelles et en g\'{e}nie du Canada (CRSNG), [num\'{e}ro de r\'{e}f\'{e}rence RGPIN-2018-04418].
Funding from the \emph{Mathematics for Industry 4.0} 2020F3NCPX PRIN2020 (MM and MZ) funded by the Italian MUR,  the  \emph{Geometric-Analytic Methods for PDEs and Applications} 2022SLTHCE (MM) and the \emph{Innovative multiscale approaches, possibly based on Fractional Calculus, for the effective constitutive modeling of cell mechanics, engineered tissues, and metamaterials in Biomedicine and related fields} P2022KHFNB (MZ) projects funded by the European Union -- Next Generation EU  within the PRIN 2022 PNRR program (D.D. 104 - 02/02/2022) is gratefully acknowledged. 
This manuscript reflects only the authors’ views and opinions and the Ministry cannot be considered responsible for them.


\end{document}